\documentclass[12pt]{article}

\usepackage[T1]{fontenc}
\usepackage{avant}
\usepackage[cp1250]{inputenc}
\usepackage{amsmath,amssymb,amsthm}
\usepackage{amscd}
\usepackage{mathtools} 
\usepackage{stmaryrd} 

\usepackage{tikz}
\usepackage[all]{xy}
\newtheorem{theorem}{Theorem}
\newtheorem{proposition}{Proposition}
\newtheorem{lemma}{Lemma}
\newtheorem{definition}{Definition}

\newtheorem{remark}{Remark}

\title{Totally non cohomologous to zero fibrations and amenable Clifford-Klein forms of 3-symmetric spaces}

\author{Maciej Boche\'nski and Aleksy Tralle}

\begin{document}
\maketitle{}
\begin{abstract}We prove that 3-symmetric spaces of simple linear real Lie groups do not admit amenable compact Clifford-Klein forms. Our basic tool are totally non cohomologous to zero fibrations.
\end{abstract}

\section{Introduction}\label{sec:intro}
Let $G/H$ be a non-compact homogeneous space of reductive type. We say that $G/H$ admits a compact Clifford-Klein form if there exists a discrete subgroup $\Gamma\subset G$ such that $\Gamma$ acts properly on $G/H$ by left translations and the quotient $\Gamma\setminus G/H$ is compact. The problem of finding compact Clifford-Klein forms is a difficult and deep problem, which appeared as a relativistic space form problem in \cite{CM}, and was formulated as a research program in \cite{Kob}. Many approaches and partial results were obtained \cite{MA},\cite{OH},\cite{OW},\cite{SH},\cite{Th}. However, it is still open even for pseudo-Riemannian symmetric spaces  \cite{KY}. Also, it is important to understand when homogeneous spaces with other invariant geometric structures admit compact Clifford-Klein forms. For example, it was shown in \cite{KY} that {\it para-Hermitian} homogeneous  spaces do not admit them. In this respect, our questions are inspired by this result. In the same vein, Benoist and Labourie \cite{BL} derived a necessary condition for the existence of compact  Clifford-Klein forms of {\it symplectic} homogeneous spaces.   At this point we want to stress that there is yet another important family of homogeneous spaces which appear in differential geometry, the class of {\it 3-symmetric spaces} determined by automorphisms of $G$ of order 3 \cite{Kow, WG} (see subsection \ref{subsec:3-sym} for some details). The latter has many interesting geometric properties, for example, all such spaces carry a canonical almost complex structure. Thus, it is natural to ask about the possible compact Clifford-Klein forms of this class, again, considering it as a ``relativistic space form problem''.  It is important to observe, that the possible algebraic structure of $\Gamma\subset G$ yielding compact Clifford-Klein forms,  is not understood. The only complete result in this direction is a theorem of Benoist \cite{BEN} which says that no nilpotent $\Gamma$ can act properly and co-compactly on semisimple $G/H$. Following this line of thinking we want to understand {\it what algebraic properties do have subgroups $\Gamma$ which yield compact Clifford-Klein forms?} In \cite{BT1} we showed that homogeneous spaces $G/H$ of parabolic type do not admit solvable Clifford-Klein forms. In \cite{BT} we analyzed the case of pseudo-Riemannian symmetric spaces and obtained some partial results. The aim of this article is to prove that in the case of $3$-symmetric homogeneous spaces of non-compact type, the problem can be fully solved. The main result of this work is formulated as follows.
\begin{theorem}\label{thm:main} 3-Symmetric homogeneous spaces $G/H$ determined by  non-compact simple real linear Lie groups $G$ and a non-compact subgroup $H$ do not admit amenable compact Clifford-Klein forms.
\end{theorem} 
\begin{remark} {\rm Note that our question on the algebraic structure of $\Gamma$ makes sense only for non-compact $H$, because for compact $H$ solvable (and, hence, amenable) compact Clifford-Klein forms do not exist for obvious reasons.}
\end{remark}
Finally, we want to stress that our main tool in this work are topological  obstructions to compact Clifford-Klein forms found by Kobayashi and Ono \cite{KO} and further developments by Morita \cite{M}. 

\section{Preliminaries}\label{sec:prelim}
\subsection{Notation}
In this work we consider only cohomology over the reals. Therefore, we omit the notation for the field of coefficients and write $H^*(X)$ for the cohomology algebra of any topological space $X$ (in fact, we consider manifolds and the de Rham cohomology). As usual, $H_*(X)$ is used for the homology of $X$. 
Throughout this paper we use the basics of Lie theory without further explanations. One can consult \cite{K}. We denote Lie groups by $G, H,...$, and their Lie algebras by the corresponding Gothic letters $\mathfrak{g},\mathfrak{h}...$.  The symbol $\mathfrak{g}^{\mathbb{C}}$ denotes the complexification of a real Lie algebra $\mathfrak{g}$. 
 Also, given a real Lie algebra $\mathfrak{g}$ we consider the Lie algebra cohomology $H^*(\mathfrak{g})$. For any subalgebra $\mathfrak{a}\subset \mathfrak{g}$ we denote by $H^*(\mathfrak{g},\mathfrak{a})$ the relative Lie algebra cohomology.
 
 Recall that {\it solvmanifolds} are the homogeneous spaces of connected solvable Lie groups. 
 
 We say that a discrete group $\Gamma$ is {\it amenable} if it admits a finitely additive probability measure.
 \subsection{Clifford-Klein forms and Kobayashi's criterion}

 Let $X$ be a Hausdorff topological space and $\Gamma$ a topological group acting on $X$. We say that an action of $\Gamma$ on $X$ is {\it proper} if for any compact subset $S\subset X$ the set
  $$\{\gamma\in\Gamma\,|\,\gamma(S)\cap S\not=\emptyset\}$$
  is compact. In particular, this article is devoted to the proper actions of a discrete subgroup $\Gamma\subset G$ on $G/H$ by left translations. 
  
 Recall that we say that $G/H$ {\it admits a compact Clifford-Klein form} if there exists a discrete subgroup $\Gamma\subset G$ acting properly and co-compactly on $G/H$. We say that a compact Clifford-Klein form is {\it solvable} (respectively, {\it amenable}), if $\Gamma$ is solvable (amenable).
 We consider homogeneous spaces $G/H$ of {\it reductive type}, which means, by definition, that $G$ is semisimple linear Lie group and there exists a Cartan involution $\theta$ on $\mathfrak{g}$ such that $\theta(\mathfrak{h})\subset\mathfrak{h}$ and that the connected Lie subgroup corresponding to $\mathfrak{h}^{\mathbb{C}}$ is closed in $Int(\mathfrak{g}^{\mathbb{C}})$. Typical examples of spaces of reductive type are semisimple symmetric and $3$-symmetric spaces.
  Using the Cartan involution, one can obtain for spaces of reductive type the {\it compatible Cartan decomposition}
  $$\mathfrak{h}=\mathfrak{k}_H+\mathfrak{p}_H, \ \mathfrak{k}_H=\mathfrak{k}\cap\mathfrak{h}, \ \mathfrak{p}_H=\mathfrak{p}\cap\mathfrak{h},$$  on the Lie algebra level and one can choose a maximal abelian subalgebra in $\mathfrak{p}_{H}$ so that $\mathfrak{a}_{H}\subset \mathfrak{a}$. 
	
   For any connected Lie group $U$ put $d(U):=dim(U)-dim(K_{U}),$ where $K_{U}$ denotes a maximal compact subgroup of $U.$ If $G/H$ is of reductive type, one can write $d(G)=\dim\mathfrak{p}$ (then, clearly, $d(H)=\dim\mathfrak{p}_H$).
We will need the following result.

\begin{theorem}[\cite{Kob},\cite{OW}]\label{thm:kob-crit}
If a closed connected subgroup $U$ acts properly and co-compactly on $G/H$ then
$$d(G)=d(H)+d(U)$$
\end{theorem}
Recall that a homogeneous space $G/H$ is {\it reductive}, if $\mathfrak{g}$ admits a  direct $\operatorname{Ad}(H)$-invariant decomposition 
$$\mathfrak{g}=\mathfrak{h}+\mathfrak{m} \eqno (1)$$
(thus a space of reductive type is reductive). For example, if $\sigma\in\operatorname{Aut}(\mathfrak{g})$ is an automorphism of $\mathfrak{g}$ of finite order, then there is a canonical reductive decomposition $(1)$ with $\mathfrak{h}=\mathfrak{g}^{\sigma}$ and $\mathfrak{m}=A_{\sigma}\mathfrak{g}$, where $A_{\sigma}=\sigma-I$. Here and throughout this paper $I$ means the identity map. For any homogeneous space $G/H$ we denote by $\Omega^*(G/H)^G$ the (differential) graded algebra of all $G$-invariant differential forms on $G/H$. It is well known \cite{GHV},\cite{KO} that for the reductive homogeneous space $G/H$ there is a natural isomorphism
$$\Omega^*(G/H)^G\cong (\Lambda\mathfrak{m}^*)^H.$$
In the latter formula, $H$ acts on $\Lambda\mathfrak{m}^*$ by the coadjoint representation, and $(\Lambda\mathfrak{m}^*)^H$ denotes the subalgebra of the fixed points of $H$. 
The (de Rham) cohomology algebra over the reals of a {\it compact} homogeneous space $G/H$ is given by the formula
$$H^*(G/H)=H^*(\Omega^*(G/H)^G).$$ 
  
\subsection{$3$-Symmetric spaces}\label{subsec:3-sym}
The importance of {\it symmetric} spaces in geometry and analysis is well known. These are homogeneous spaces $G/H$, where $H$ is an open subgroup of the subgroup $G^{\sigma}$ of all fixed points of an {\it involutive} automorphism $\sigma\in\operatorname{Aut}(G)$.  
In this article we consider  {\it 3-symmetric spaces}, which are defined by a similar condition 
$$G^{\sigma}_0\subset H\subset G^{\sigma}$$
as for symmetric spaces, but with the requirement that $\sigma$ is an automorphism of $G$ of order 3. These spaces were classified by Wolf and Gray \cite{WG}. The geometric applications of this class are concerned with $G$-invariant almost complex structures and almost Hermitian geometry. The homogeneous spaces $G/H$ of reductive type generated by automorphisms of order 3 carry natural almost complex structures, because Wolf and Gray proved the following.
\begin{theorem}\label{thm:wg} Let $G/H$ be a (reductive) homogeneous 3-symmetric space. Consider the reductive decomposition 
$$\mathfrak{g}=\mathfrak{h}+\mathfrak{m}, \,\mathfrak{m}=A_{\sigma}\mathfrak{g}.$$
Then the canonical $G$-invariant almost complex structure on $G/H$ is determined by the formula
$$P:\mathfrak{m}\rightarrow\mathfrak{m},\,P=-{1\over 2}I+{\sqrt{3}\over 2}J,$$
where $P$ denotes the induced action $\sigma|_{\mathfrak{m}}$.
\end{theorem}
Thus one can consider $G$-invariant almost Hermitian metrics on $G/H$ and investigate important properties of being Hermitian, semi-Kaehler, almost Kaehler and nearly Kaehler. All these classes are of importance to geometry and are used in physical applications. Let us just mention the relations of 3-symmetric spaces to twistor theory \cite{Bu}, Ricci solitons \cite{CR} and symplectic geometry \cite{BT2}. It seems that the relation of an important class of homogeneous nearly Kaehler manifolds to 3-symmetric spaces deserves a special attention \cite{C}.   
  
 \subsection{Kobayashi and Ono theorem} 
 Let $G/H$ be a  homogeneous space of reductive type. Let $G_{u}$ be a compact real form of a (connected) complexification $G^{\mathbb{C}}$ of $G$ and let $H_{u}$ be a compact real form of $H^{\mathbb{C}}\subset G^{\mathbb{C}}$. The space $G_{u}/H_{u}$ is called the  homogeneous space of compact type associated with $G/H$ (or dual to $G/H$). Groups $G_{u},$ $H_{u}$ are called the compact duals of $G$ and   $H,$ respectively.
\begin{theorem}[\cite{KO}, Theorem 4]\label{thm:ko} Assume that a discrete subgroup $\Gamma \subset G$ acts on $G/H$ freely and properly. Then there is a natural map
$$\eta: H^*(G_{u}/H_{u})\rightarrow H^*(\Gamma\setminus G/H)$$
which sends the Pontryagin classes of $G_{u}/H_{u}$ to those of $\Gamma\setminus G/H$. If $\Gamma\setminus G/H$ is compact then $\eta$ is injective. Also, it sends the Euler class of $G_u/H_u$ to the Euler class of $\Gamma\setminus G/H$.
\label{ko}
\end{theorem}
\subsection{Syndetic hull}\label{subsect:syndetic}
We will need the notion of a syndetic hull \cite{W}.
\begin{definition}
{\rm A {\it syndetic hull} of a subgroup $\Gamma$ of a Lie group $G$ is a subgroup $B$ of $G$ such that $B$ is connected, $B$ contains $\Gamma$ and $\Gamma\setminus B$ is compact.}
\end{definition}

\begin{lemma}[\cite{BT1}]\label{lemma:syndetic} If a solvable subgroup $\Gamma\subset G$ acts properly and co-compactly on $G/H$, then 
\begin{enumerate}
\item there exists a syndetic hull of $\Gamma$,
\item there exists a solvable subgroup $\Gamma_{0}\subset AN$ that acts properly and co-compactly on $G/H$, with a syndetic hull $B\subset AN$.
\end{enumerate}
\end{lemma}

\section{TNCZ-fibrations and cohomology of homogeneous spaces}
In this section we discuss our basic topological tool \cite{GHV}, \cite{FHT},\cite{FOT}. One can find all the necessary information about the role of the totally non cohomologous to zero fibrations, differential graded algebras, and so on, in these sources. We use this material without further explanations.
\subsection{Totally non cohomologous to zero fibrations}
Let 
$$
\CD
F @>{i}>> E @>{\pi}>> B
\endCD
$$
be a Serre fibration. We say that it is {\it totally non-cohomologous to zero} (TNCZ for short) if one of the following equivalent conditions is satisfied:
\begin{itemize}
\item $i^*: H^*(E)\rightarrow H^*(F)$ is onto,
\item $\pi^*: H^*(B)\rightarrow H^*(E)$ in injective,
\end{itemize}
It appears that the TNCZ condition yields and obstruction to compact Clifford-Klein forms. This was discovered   by Morita \cite{M}. We recall his result  in a slightly different form. 
\begin{theorem}\label{thm:tncz} Assume that a homogeneous space $G/H$ of reductive type admits a compact Clifford-Klein form. Then the fiber bundle
$$H_U/K_H\rightarrow G_u/K_H\rightarrow G_u/H_u$$
must be a TNCZ-fibration.
\end{theorem}
\begin{proof} Recall the isomorphisms from \cite{KO}:
$$\Omega^*(G/H)^G\cong (\Lambda\mathfrak{q}^*)^H$$
$$\Omega^*(G_u/H_u)^{G_u}\cong (\Lambda\mathfrak{q}_u^*)^{H_u}$$
One can see from the proof of these isomorphisms, that the same holds for the decompositions
$$\mathfrak{g}=\mathfrak{k}_h+\tilde{\mathfrak{q}},\,\mathfrak{g}_u=\mathfrak{k}_h+\tilde{\mathfrak{q}}\cap\mathfrak{k}+
\sqrt{-1}\tilde{\mathfrak{q}}\cap\mathfrak{p}$$
and, moreover, the inclusions hold:
$$\mathfrak{k}_h\subset\mathfrak{k},\mathfrak{q}\subset\tilde{\mathfrak{q}},\,\mathfrak{q}_u\subset\tilde{\mathfrak{q}}_u.$$
It follows that the following diagram is commutative
$$
\CD
(\Lambda\tilde{\mathfrak{q}}^*)^{K_H}) @>{\cong}>>(\Lambda\tilde{\mathfrak{q}}_u^*)^{K_H}\\
@A{incl}AA @A{incl}AA\\
(\Lambda\mathfrak{q}^*)^H @>>> (\Lambda\mathfrak{q}_u^*)^{H_u}
\endCD
$$
As a result we get the commutative diagram
$$
\CD
\Omega^*(G_u/K_H)^{G_u} @>{\hat{\varphi}}>>\Omega^*(G/K_H)^G @>{incl}>>\Omega^*(\Gamma\setminus G/H)\\
@A{\pi_1}AA @A{\pi_2}AA\\
\Omega^*(G_u/H_u)^{G_u} @>{\varphi}>>\Omega^*(G/H)^G @>{incl}>>\Omega^*(\Gamma\setminus G/H)
\endCD
$$
Here $\pi_1$ and $\pi_2$ denote the pullbacks. The latter yields the commutative diagram
$$
\CD
H^*(G_u/K_H) @>{\hat{\gamma}}>>H^*(\Gamma\setminus G/K_H)\\
@AAA @AAA\\
H^*(G_u/H_u) @>{\gamma}>> H^*(\Gamma\setminus G/H)
\endCD
$$
Now, assuming that $\Gamma\setminus G/H$ is compact we conclude that $\gamma$ is injective \cite{KO}. But since $H^*(\Gamma\setminus G/H)\cong H^*(\Gamma\setminus G/K_H)$, we see from the commutativity that the left vertical arrow must be injective as well. The proof is complete.
\end{proof}
 
\subsection{Cohomology of compact homogeneous spaces} 

In this section we recall the classical results on real cohomology of compact homogeneous spaces in the suitable for us form \cite{T}. Let $G$ be a compact connected Lie group. The cohomology algebra $H^*(G)$ is an exterior algebra over vector space generated by primitive elements $y_1,...,y_n$ of odd degrees $2k_1-1,...,2k_n-1$, where $n={rank}\,G$. These generators can be chosen in a way, that the cohomology algebra of the classifying space $BG$ is isomorphic to the polynomial algebra $\mathbb{R}[x_1,...,x_n]$, where elements $y_j$ are cohomology classes in the universal space $EG$ containing coboundaries of the transgression for $x_j$. Let $H\subset G$ be a closed subgroup. Choose maximal tori $T_H\subset T$ of $H$ and $G$, and their Lie algebras $\mathfrak{t}_H\subset\mathfrak{t}$. It is well known that the cohomology algebras $H^*(BG)$ and $H^*(BH)$ are isomorphic to the polynomial algebras of the invariants of the Weyl groups $W(G)$ and $W(H)$ respectively. It is a classical result that the canonical homomorphism $\rho^*(H,G): H^*(BG)\rightarrow H^*(BH)$ induced by the inclusion $H\hookrightarrow G$ sends any invariant polynomial from $H^*(BG)$ its restriction on $H^*(BH)$. 
Introduce the following notation:
$$H^*(BG)\cong\mathbb{R}[\mathfrak{t}]^{W(G)}=\mathbb{R}[f_1,...,f_n],$$
$$H^*(BH)\cong \mathbb{R}[\mathfrak{t}_H]^{W(H)}=\mathbb{R}[u_1,...,u_r],$$
where $f_j$ are $W(G)$-invariant polynomials over $\mathfrak{t}$, and $u_k$ are $W(H)$-invariant polynomials over $\mathfrak{t}_H$.
The cohomology algebra of $G/H$ can be calculated with the use of the following.
\begin{theorem}[\cite{T}, Theorem 8]\label{thm:g/h} Assume that $f_1,...,f_n$ are generators of the ring $\mathbb{R}[\mathfrak{t}]^{W(G)}$, and assume that $r=\operatorname{rank}\,H$.  Let $j^*:\mathbb{R}[\mathfrak{t}]^{W(G)}\rightarrow \mathbb{R}[\mathfrak{t}_H]^{W(H)}$ denote the restriction map $j^*(f)=f|_{\mathfrak{t}_H}$ (which corresponds to $\rho^*(G,H)$ under the suitable isomorphisms). If $j^*(f_{r+1}),...,j^*(f_n)$ belong to the ideal $(j^*(f_1),...,j^*(f_r))$ generated by polynomials $j^*(f_1),...,j^*(f_r)$, then
$$H^*(G/H)=(\mathbb{R}[\mathfrak{t}_H]^{W(H)}/(j^*(f_1),...,j^*(f_r))\otimes\Lambda (x_{r+1},...,x_n).$$
\end{theorem}
As usual, the differential graded algebra
$$(C_{G/H},d)=(H^*(BH)\otimes \Lambda (y_1,...,y_n),d)$$
$$d|_{H^*(BH)}=0,\,dy_j=j^*(f_j),\,j=1,...,n$$
is called the {\it Cartan algebra} of $G/H$. Note that $H^*(G/H)\cong H^*(C_{G/H})$ for any homogeneous space $G/H$ of compact Lie group $G$.

\section{Proof of the main theorem}
The proof of the main theorem follows from the statements below.
\begin{proposition}\label{prop:d} Assume that $d=\dim\,H_u/K_H$. If the coefficient in degree $d$ of the  Poincar\'e polynomial $P(G_u/K_H,t)$ vanishes, then $G/H$ does not admit compact Clifford-Klein forms.
\end{proposition}
\begin{proof} Recall that if $G/H$ admitted a compact Clifford-Klein form, one would have $i_*[H_u/K_H]\not=0$, by Theorem \ref{thm:tncz} and since $d=\dim (H_u/K_H)$.  On the other hand, one can see that $i=p\circ j$, where $j$ and $p$ are determined by the sequence
$$
\CD
H_u/K_H @>{j}>> G_u/K_H @>{p}>> G_u/K
\endCD
$$
By assumption on the Poincar\'e polynomial of $G_u/K_H$, $H_d(G_u/K_H)=0$.
Hence, $i_*[H_u/K_H]=p_*(j_*[H_u/K_H])=0$, because $j_*[H_u/K_H]\in H_d(G_u/K_H)=0$.  
\end{proof}

\begin{proposition}\label{prop:equal-rank} Let $G/H$ be a non-compact homogeneous space of reductive type. Assume that $\operatorname{rank}\,G=\operatorname{rank}\,H$. Then no solvable $\Gamma$ can yield a compact Clifford-Klein form of $G/H$.
\end{proposition}
\begin{proof} Without loss of generality we assume that $\Gamma$ acts properly and freely on $G/H$. Note that $G/K_H$ and $G/H$ have the same homotopy type, and the same holds for $\Gamma\setminus G/K_H$ and $\Gamma\setminus G/H$. Let $\chi(\Gamma)=\chi(B\Gamma)=\chi(\Gamma\setminus G/K)$. Looking at the fibration
$$K/K_H\rightarrow \Gamma\setminus G/K_H\rightarrow \Gamma\setminus G/K$$
one obtains 
$$\chi(\Gamma\setminus G/H)=\chi(\Gamma\setminus G/K_H)=\chi(K/K_H)\cdot\chi(\Gamma).$$
By assumption of the Proposition, $\operatorname{rank}\,G = \operatorname{rank}\,H$ which implies $\operatorname{rank}\,G_u=\operatorname{rank}\,H_u$ and, therefore, $\chi(G_u/H_u)\not=0$. Thus, the Euler class of $T(G_u/H_u)$ does not vanish, and by Theorem \ref{thm:ko}, $e(T(\Gamma\setminus G/H)\not=0$, which implies that the Euler characteristic of  $\Gamma\setminus G/H$ does not vanish. On the other hand, since $\Gamma$ admits a syndetic hull $B$ by Lemma \ref{lemma:syndetic}, one obtains a compact solvmanifold $\Gamma\setminus B$.  Since $\Gamma$ contains a subgroup of finite index which can be embedded into a simply-connected solvable Lie group $\tilde B$ as a lattice (see \cite{VGS}, Theorem 3.14), one obtains $\chi(\Gamma)=\chi(\Gamma\setminus\tilde B)$. It is known that the Euler characteristic of a compact solvmanifold vanishes \cite{HIL}. We have arrived at a contradiction.
\end{proof}
Note again that everywhere $B$ denotes the syndetic hull of $\Gamma$ which exists by Lemma \ref{lemma:syndetic}. Consider the relative Lie algebra cohomology $H^*(\mathfrak{g},\mathfrak{a},\mathbb{R})$ (for any subalgebra $\mathfrak{a}\subset\mathfrak{g}$). 

Denote by $\Gamma_H$ a co-compact lattice in $H$. It is straightforward that $M_B=\Gamma_H\setminus G/B$ is a compact manifold of dimension, say, $n$. Since $B$ is a syndetic hull of $\Gamma$, it is unimodular, which implies $(\Lambda^n(\mathfrak{g}/\mathfrak{b}^*)^B\not=0.$ It follows that there exists an invariant $n$-form $\Phi$ on $G/B$ which yields a volume form on $M_B$. It follows that $H^n(\mathfrak{g},\mathfrak{b},\mathbb{R})\not=0$. Let $A\subset B$ be a closed connected subgroup and $\mathfrak{a}\subset\mathfrak{b}$ the corresponding Lie subalgebra. Denote by $\eta$ the natural map in the relative Lie algebra cohomology $H^*(\mathfrak{g},\mathfrak{b})\rightarrow H^*(\mathfrak{g},\mathfrak{a}).$
\begin{proposition}\label{prop:Lie-cohom} If $G/H$ admits a solvable compact Clifford-Klein form, then the map
$$
\CD
0\not=H^n(\mathfrak{g},\mathfrak{b}) @>{\eta}>> H^n(\mathfrak{g},\mathfrak{a})
\endCD
$$
is injective.
\end{proposition}
\begin{proof} Consider the commutative diagram
$$
\CD
H^n(\mathfrak{g},\mathfrak{b}) @>{\eta}>> H^n(\mathfrak{g},\mathfrak{a})\\
@V{f}VV @V{g}VV\\
H^n(M_B) @>{\pi^*}>> H^n(\Gamma_H\setminus G/A)
\endCD
$$
where $\pi: \Gamma_H\setminus G/A\rightarrow \Gamma_H\setminus G/B$ denotes the natural projection. Note that $\pi$ has contractible fiber, and $f$ is injective, which implies the injectivity of $\eta$.
\end{proof} 
\begin{proposition}\label{prop:so(4,4)} The $3$-symmetric spaces 
$$SO(4,4)/SU(1,2),\,SO(4,4)/G_{2(2)}$$
do not admit solvable compact Clifford-Klein forms.
\end{proposition}
\begin{proof} By Proposition \ref{prop:Lie-cohom} the map 
$$H^n(\mathfrak{a},\mathfrak{b})\rightarrow H^n(\mathfrak{g})$$
is injective in degree $n=16$ for $SO(4,4)/SU(1,2)$ and in degree $n=20$ for the case $SO(4,4)/G_{2(2)}$. 
The latter follows from Theorem \ref{thm:kob-crit}. Indeed, one can see that
$\dim\mathfrak{p}=16$ for $\mathfrak{g}=\mathfrak{so}(4,4)$ and $\dim\mathfrak{p}_H=8$ for $\mathfrak{h}=\mathfrak{g}_{2(2)}$. Since $\Gamma\setminus G/H$ is compact, Theorem \ref{thm:kob-crit} implies $\dim\mathfrak{b}=8$. Hence $\dim (SO(4,4)/B)=20$.
By Lemma 3.7 and Proposition 3.9 in \cite{KO} $H^*(\mathfrak{g})$ is isomorphic to $H^*(\mathfrak{g}_u)$. Here $\mathfrak{g}_u=\mathfrak{s}\mathfrak{o}(8)$, which yields $\dim\mathfrak{g}_u=28$, and the following degrees of the primitive elements in the Lie algebra cohomology: $3, 7, 11, 7$.  This necessarily implies $H^{16}(\mathfrak{g})=0$ and $H^{20}(\mathfrak{g})=0$, a contradiction.   
\end{proof}
\begin{proposition}\label{prop:g2} The 3-symmetric space $SO(3,5)/G_{2(2)}$ does not admit compact Clifford-Klein forms.
\end{proposition}
\begin{proof}
 We will show that $i_*[H_u/K_H]$ is zero in $H_*(G_u/K)$. Then the proof will follow from Theorem \ref{thm:tncz}.
 In our case
 $$G_u=SO(8), K=SO(3)\times SO(5), H_u=G_2, K_H=H_u\cap K=SO(3)\times SO(3).$$
Note that $d=\dim (G_2/SO(3)\times SO(3))=8.$  Hence the fundamental homology class $[H_u/K_H]=[G_2/(SO(3)\times SO(3))]\in H_8(H_u/K_H)$. 
To show the required equality, we  plainly calculate the cohomology algebra of $G_u/K_H$, calculating the corresponding Cartan algebra. 
 First, we need to calculate the Cartan subalgebra of $\mathfrak{t}_H$ and the embedding $\mathfrak{t}_H\subset\mathfrak{t}$. Here we can use the calculation from \cite{T} on pages 150-151 with the following modification. Note that $\operatorname{rank}\,H_u=\operatorname{rank}\,K_H$. It follows that one can choose the same Cartan subalgebra $\mathfrak{t}_H$ for $\mathfrak{k}_H$ and $\mathfrak{h}_u$. The following was shown in \cite{T}. Let $\mathfrak{g}_u$ be a simple compact Lie algebra of type $D_4$.  Then, any Cartan subalgebra of a subalgebra $\mathfrak{h}_u\subset\mathfrak{g}_u$ generated by an {\it outer automorphism} can be described by the formula
 $$\mathfrak{t}'_H=\{ (x_1,...,x_4)\,|\,x_1=0, x_2 + x_3=x_4\}.$$
 Note that Theorem 5.5 in \cite{WG} implies that $\mathfrak{h}_u$ is determined by an outer automorphism of $\mathfrak{g}_u$, therefore, we can use the above formula. 
This justifies the calculation below.
We get
$$\mathbb{Q}[\mathfrak{t}]^{W(SO(8))}=\mathbb{Q}[f_1,...,f_4]$$
where
$$f_1=x_1^2+x_2^2+x_3^2+x_4^2,\,
 f_2=x_1^2x_2^2+x_1^2x_3^2+x_1^2x_4^2+x^2_2x_3^2+x_2^2x_4^2+x_3^2x_4^2$$
$$f_3=x_1^2x_2^2x_3^2+x_1^2x_2^2x_4^2+x_1^2x_3^2x_4^2+ x_2^2x_3^2x_4^2,\,f_4=x_1x_2x_3x_4.$$
In the same way
$$\mathbb{Q}[\mathfrak{t}_H]^{W(SO(3)\times SO(3))}=\mathbb{Q}[x_2^2,x_3^2].$$
Finally, this yields
$$H^*(SO(8)/SO(3)\times SO(3))=\mathbb{Q}[x_2^2,x_3^2]/(g_4,g_8)\otimes\Lambda (y_7,y_{11}), \eqno (*)$$
where 
$$f_1|_{\mathfrak{t}_H'}=g_4=x_2^2+x_3^2+(x_2+x_3)^2, $$
$$f_2|_{\mathfrak{t}_H'}=g_8=x_2^2x_3^2+(x_2^2+x_3^2)(x_2+x_3)^2$$
$$f_3|_{\mathfrak{t}_H'}=g_{12}=x_2^2x_3^2(x_2+x_3)^2,$$
$$f_4|_{\mathfrak{t}_H'}=\tilde g_8=0.$$
Note that $(*)$ follows from the observation that $g_{12}\in (g_4,g_8).$
Observe the following. Since $g_4=x_2^2+x_3^2+(x_2+x_3)^2$, it follows that 
$$x_2^2+x_3^2+x_2^2+2x_2x_3+x_3^2=0 \,\,\text{in}\,\,H^*(G_u/K_H)$$
which yields $x_2^2+x_3^2=-x_2x_3$ in the cohomology.  Looking at the expression of $g_8$, one can see that $x_2^2x_3^2=0$ in $H^*(G_u/K_H)$. It follows that $H^8(G_u/K_H)=0$, and the proof follows from Proposition \ref{prop:d}.
\end{proof}
\vskip6pt
\noindent {\bf Completion of proof of the main theorem} It follows from the Tits alternative that if $G/H$ admits a compact amenable Clifford-Klein it also admits a solvable Clifford-Klein form. Thus, it is sufficient to eliminate the possibility of getting the latter. By Proposition \ref{prop:equal-rank} we should consider only $3$-symmetric spaces with simple Lie groups of non-compact type and with $\operatorname{rank}\,G>\operatorname{rank}\,H$. The latter are contained in Table 7.14 of \cite{WG}. Examining this Table we get the following list:
$$SO(4,4)/SU(1,2),\,SO(4,4)/G_{2(2)},\,SO(3,5)/G_{2(2)}, Spin(1,7)/G_2.$$
The first three cases are eliminated by Proposition \ref{prop:so(4,4)} and Proposition \ref{prop:g2}. The case of $Spin(1,7)/G_2$ should not be considered, because $G_2$ is compact. 
\hfill$\square$
\begin{remark}{\rm In Table 7.14 in \cite{WG} the authors use a bit different notation: $SO^m(n)$ from this Table corresponds to our $SO(m,n-m)$, and $G_2^*$ to $G_{2(2)}$ (see also page 118 in \cite{WG}). For obvious reasons, we exclude non-simple $G$ and complex Lie groups.}
\end{remark}

\end{document}